\documentclass[11pt,reqno]{amsart}

\usepackage{latexsym}
\usepackage{amsfonts}
\usepackage{amsmath}
\usepackage{url}
\usepackage{graphicx}
\usepackage{color}
\usepackage{bbm}
\definecolor{myblue}{rgb}{0.09,0.32,0.44} 
\usepackage[pdfborder={0 0 0},pdfborderstyle={},colorlinks=true,linkcolor=myblue,citecolor=myblue,urlcolor=blue]{hyperref}

\usepackage{epsfig,psfrag,verbatim,amsfonts}

\usepackage{amssymb,amsmath}

\def\bl{\begin{lemma}}
\def\el{\end{lemma}}
\def\bth{\begin{theorem}}
\def\eth{\end{theorem}}
\def\bc{\begin{corollary}}
\def\ec{\end{corollary}}
\def\bcj{\begin{conjecture}}
\def\ecj{\end{conjecture}}
\def\bpr{\begin{proposition}}
\def\epr{\end{proposition}}
\def\bde{\begin{definition}}
\def\ede{\end{definition}}

\def\H{\mathbb{H}}
\def\Pr{\mathbb{P}}

\newcommand{\be}{\begin{eqnarray}}
\newcommand{\ee}{\end{eqnarray}}

\newcommand{\R}{{\mathbb R}}

\renewcommand{\and}{\hbox{ {\rm and} }}

\newtheorem{theorem}{Theorem}[section]
\newtheorem{definition}{Definition}[section]
\newtheorem{lemma}[theorem]{Lemma}

\newtheorem{corollary}[theorem]{Corollary}
\newtheorem{proposition}[theorem]{Proposition}
\newtheorem{conjecture}[theorem]{Conjecture}

\newtheorem*{theorem*}{Theorem}

\theoremstyle{definition}
\numberwithin{equation}{section}

\begin{document}
\title{Hyperbolic self avoiding walk}
\author{Itai Benjamini \and Christoforos Panagiotis}

\date{25.3.2020}

\maketitle

\begin{abstract}
Expected ballisticity of a continuous self avoiding walk on hyperbolic spaces $\H^d$ is established. 

\end{abstract}

\section{Introduction}

Consider continuous $n$ steps random walk on $\H^d$, where the next step is chosen uniformly and independently on the unit sphere around the current location. Condition this uniform product measure
on sequences in which the distance between any pair of vertices is bigger than $c$, for some fixed $0 < c < 1$. 

We will write $\Pr_n$ for this measure, which will be called the SAW measure, and $\mathbb{E}_n$ for the expectation with respect to this measure. We denote by $x_0,x_1,\ldots,x_n$ the vertices of the walk in order of appearance and $d$ the distance function.

\begin{theorem}\label{main}
There exists a constant $C=C(d,c)>0$ such that $\mathbb{E}_n d(x_0,x_n) > Cn$ for every $n\geq 1$. 
\end{theorem}

For background on hyperbolic geometry see e.g. \cite{BI}. For background on self avoiding walks see
\cite{BDGS}. The papers \cite{DH, L,MW,P} contains results of the speed of self avoiding walks.

\section{Proof}

\begin{proof}

Let us first recall that there is a constant $\delta$ fixed throughout such that for every geodesic triangle in $\H^d$ with sides $\alpha,\beta,\gamma$ we have that $\alpha$ lies in the $\delta$ neighbourhood of $\beta\cup \gamma$. 

We will start by proving the following geometric lemma. 

\begin{lemma}\label{convex}
There is a constant $c_1>0$ such that for every finite set $A$ of vertices in $\H^d$ with the property that the distance between any two of its points is at least $c$, at least $c_1|A|$ vertices of $A$ are at distance at most $(1-c)/2$ from the boundary of the convex hull of $A$.
\end{lemma}
\begin{proof}
Notice that the convex hull $K$ of $A$ coincides with the convex hull of $A\cap \partial K$. According to \cite{BE}, there is a constant $t=t(d)$ such that the volume of $K$ is at most $t|A\cap \partial K|$. For every $x$ in $B:=\{x\in A \mid d(x,\partial K)>(1-c)/2\}$, the ball of radius $(1-c)/2$ around $x$ is contained in $K$, hence the volume of $K$ is greater than $t'|B|$ for a certain constant $t'>0$. The assertion follows now easily.
\end{proof}

For a set of vertices $A$, we will write $\mathcal{H}(A)$ for the convex hull of $A$. It follows from Lemma \eqref{convex} that there is a constant $R>0$ such that for every SAW $x$ of length $n$, the number of indices $i$ with $d(x_i,\partial \mathcal{H}(x)),d(x_j,\partial \mathcal{H}(x))\leq (1-c)/2$ for some $i+1\leq j\leq i+R$ is of order $n$. As there are finitely many choices for $j-i$, the number of indices $i$ with $d(x_i,\partial \mathcal{H}(x)),d(x_j,\partial \mathcal{H}(x))\leq (1-c)/2$ is of order $n$, where $1\leq r=r(A)\leq R$ is the number that appears most often. In fact, we can assume that $r$ is as large as we want. In particular, we can always choose $r>3\delta+1-c$. The reason for making this choice will become clear later.

For any number $\mathcal{C}>0$, we define $\mathcal{A}_i(\mathcal{C})$ to be the event that $x_i$ has distance at most $\mathcal{C}$ from the geodesic between $x_0$ and $x_n$. We also define $\mathcal{B}_i$ to be the event $d(x_i,\partial \mathcal{H}(x)),d(x_j,\partial \mathcal{H}(x))\leq (1-c)/2$. Our aim is to utilize the above observation in order construct SAWs that satisfy $\mathcal{A}_i(\mathcal{C})$.

\begin{lemma}\label{RD}
There are universal constants $0<C<1$, $\mathcal{C}>0$ such that $\Pr_n(\mathcal{A}_i(\mathcal{C}))\geq C\Pr_n(\mathcal{B}_i)$ for every large enough $m$ and every $i\geq 1$.
\end{lemma}
\begin{proof}
Consider a SAW $x\in \mathcal{B}_i$. Our aim is to define a family of walks $x'$ for which the position of $x'_0,\ldots,x'_i$ and the relative position between the points $x'_{i+r},\ldots,x'_n$ is fixed and coincides with that for $x$, and furthermore, the event $\mathcal{A}_i(\mathcal{C})$ occurs.

To this end, let $z_i$, $z_{i+r}$ be the points in the boundary of $\mathcal{H}(x_0,\ldots,x_i)$, $\mathcal{H}(x_{i+r},\ldots,x_n)$ closest to $x_i$, $x_{i+r}$, respectively, and write $d$ for the distance between $z_i$ and $z_{i+r}$. Consider two separating hyperplanes $H_1$ and $H_2$ that pass through $z_i$ and $z_{i+r}$, respectively. In what follows, the position of $z_{i+r}$ and $x_{i+r},\ldots,x_n$ will change but for convenience we will keep the same notation throughout the proof.

Notice that we can rotate $\mathcal{H}(x_{i+r},\ldots,x_n)$ around $z_{i+r}$ and then around $z_i$ so that $H_2$ lies in the complement of the open ball of radius $d$ around $z_i$ and has distance $d$ from $H_1$. Now we can apply another isometry to $H_2$, if necessary, so that $H_1$ and $H_2$ are at distance $$r-d(x_i,z_i)-d(x_{i+r},z_{i+r})\geq r+c-1> 3\delta$$ apart and the distance is attained by the pair $z_i$, $z_{i+r}$. Adding the geodesic walk $y_i=x_i,\ldots,y_{i+r}=x_{i+r}$ of length $r$ that connects $x_i$ to $x_{i+r}$ results in a SAW of length $n$ because any $y_{i+1},\ldots,y_{i+r-1}$ is at distance at least $1-(1-c)/2>c$ from $H_1$, $H_2$. We can now construct a family of walks with the desired properties by choosing the $j$th step, $j=i+1,\ldots,i+r$ suitably, while keeping the distance between $H_1$ and $H_2$ close to $r-d(x_i,z_i)-d(x_{i+r}$ and larger than $3\delta$. In fact, these steps can be chosen from some domes $D_j$, $j=i,\ldots,i+r$ of positive area that differ only by a rotation. 

It follows from the next lemma that all these SAWs satisfy $\mathcal{A}_i(\mathcal{C})$. Since the domes $D_j$ have positive measure, the probability of $\mathcal{A}_i(\mathcal{C})$ conditioned on $x_0,\ldots,x_i$ and the relative position between the points $x_{i+r},\ldots,x_n$ is bounded from below by a uniform positive constant, whenever $x\in \mathcal{B}_i$. Taking expectation we obtain the desired assertion.
\end{proof}

In the following lemma, we will write $\overline{xy}$ for the geodesic between points $x$ and $y$ in $\H^d$ and $d(A,B) := \inf_{x \in A, y \in B} d(x,y) \geq \delta$ for the distance between subsets of $\H^d$.

\begin{lemma}\label{two geodesics}
Consider two bi-infinite geodesics $A$, $B$ in $\H^d$
with $d(A,B)> 3\delta$. Assume that $x_0 \in A$ and $y_0 \in B$ realize the distance $d(A,B)$, then the following holds:
for any $x \in A, y \in B$ we have $d(\overline{xy},\overline{x_0y_0})< 2\delta$.
\end{lemma}
\begin{proof}
Let $x \in A, y \in B$ and write $\gamma$, $\gamma'$ for the geodesics starting at $x_0$ and ending at $x$, $y$, respectively, parametrized by arc-length. 

Consider an $R>0$ such that $d(\gamma(R),\gamma'(R))= 2\delta$. We claim that $d(\gamma'(R),\overline{xy})\leq \delta$. Indeed, suppose to the contrary that $d(\gamma'(R),\overline{xy})> \delta$. Since $d(\gamma'(R),\gamma\cup \overline{xy})<\delta$, there exists some $t\in \gamma$ such that $d(\gamma'(R),t)<\delta$. Hence $R-\delta< d(x_0,t)< R+\delta$, which implies that $d(\gamma(R),t)<\delta$. Thus $d(\gamma(R),\gamma'(R))<2\delta$ and this contradiction proves the claim.

Now for the triangle with vertices $x_0,y_0,y$ we have that 
$d(\gamma'(R),\overline{x_0y_0}\cup\overline{y_0y})<\delta$. But $d(\gamma'(R),\overline{y_0y})> \delta$ which shows that $d(\gamma'(R),\overline{x_0y_0})<\delta$ and completes the proof.
\end{proof}

Using Lemma \eqref{RD} and the fact that on any SAW of length $n$, the number of indices $i$ such that $\mathcal{B}_i$ occurs is of order $n$, we obtain that in a SAW of length $n$, the expected number of indices $i$, for which $x_i$ is within distance $\mathcal{C}$ to the geodesic between $x_0$ and $x_n$ is of order $n$. Since the open balls of radius $c/2$ around the vertices of the SAW are disjoint, we get that the $\mathcal{C}$-neighbourhood of the geodesic has expected area of order $n$. But the area of the $\mathcal{C}$-neighbourhood of a geodesic of length $k$ is of order $k$ for any $k\geq 1$. This completes the proof of the theorem.
\end{proof}

\section{Open problems}

In this section we will state some open problems. Given some $\varepsilon=\varepsilon(n)>0$, we consider continuous $n$ steps SAWs on $\H^d$, where now the steps are chosen uniformly from the sphere of radius $\varepsilon$ and we condition on sequences in which the distance between any pair of vertices is bigger than $c\varepsilon$, for some fixed $0 < c < 1$. We are interested in the behaviour of these SAWs as $\varepsilon$ tends to $0$. We will refer to the $n$ step model as the $(n,\varepsilon)$-SAW. Let us write $\mathbb{E}_{n,\varepsilon}$ for the expectation with respect to this measure.
 
\begin{conjecture}
$\lim_{n\to\infty} \mathbb{E}_{n,1/n}(d(x_0,x_n))=0$
\end{conjecture}

\begin{conjecture}
There exists a constant $1>\beta=\beta(d,c)>0$ such that $$\lim_{n\to\infty} \mathbb{E}_{n,n^{-\beta}}(d(x_0,x_n))$$ exists and is not zero (possibly up to logarithmic factors).
\end{conjecture}

We expect that $\beta$ is increasing in $c$ because when the self-avoidance restrictions are stronger, the walk tends to move further away from the origin.

Let us mention the observation motivating these conjectures. Given $\varepsilon>0$, let $\frac{1}{\varepsilon}\H^d$ denote $\H^d$ with the metric scaled by $a$. It is clear that the $(n,\varepsilon)$-SAW on $\H^d$ coincides with the $(n,1)$-SAW on $\frac{1}{\varepsilon}\H^d$. Moreover, $d(x_0,x_n)=\varepsilon d_{\varepsilon}(x_0,x_n)$, where $d_{\varepsilon}$ is the distance function on $\frac{1}{\varepsilon}\H^d$. As $\varepsilon$ tends to $0$, the curvature of $\frac{1}{\varepsilon}\H^d$ tends to $0$ as well and $\frac{1}{\varepsilon}\H^d$ looks more and more like $\mathbb{R}^d$. This indicates that for certain values of $n$ and $\varepsilon$, the $(n,\varepsilon)$-SAW on $\H^d$ behaves like $(n,1)$-SAW on $\mathbb{R}^d$. In fact, we can make a more precise prediction for the values of $n$ and $\varepsilon$ for which this holds. On balls of size smaller than $\frac{1}{\varepsilon}$, $\frac{1}{\varepsilon}\H^d$ looks similar to $\R^d$, while on balls of size larger than $\frac{1}{\varepsilon}\H^d$, the two spaces look different. If the $(n,1)$-SAW on $\mathbb{R}^d$ is typically at distance $n^{\beta}$, then this indicates that the
$(n,1)$-SAW on $n^{-\beta}\H^d$ behaves like the $(n,1)$-SAW on $\mathbb{R}^d$. This reasoning leads naturally to the above questions.

As $c$ tends to $0$, we expect that $\beta$ tends to $1/2$ because in the limit we obtain a random walk. Also, for $d\geq 5$, we expect that $\beta$ is always equal to $1/2$, as in this case, the scaling limit of SAW on $\mathbb{Z}^d$ is the Brownian motion \cite{HS}.

\end{document}